\documentclass{amsart}

\usepackage{color}
\usepackage{soul}
\usepackage{amsmath, amssymb}
\usepackage{hyperref}
\hypersetup{
    colorlinks=true, %set true if you want colored links
    linktoc=all,     %set to all if you want both sections and subsections linked
    linkcolor=blue,  %choose some color if you want links to stand out
}

\newcommand{\dotieconcat}[2]{% auxiliary macro, don't use it directly
  \text{\raisebox{.8ex}{$\smallfrown$}}%
}

\newcommand{\Q}{\mathbb{Q}}
\newcommand{\C}{\mathbb{C}}

\newcommand{\N}{\mathbb{N}}

\newcommand{\diam}{\operatorname{diam}}

\newcommand{\cstar}{$\mathrm{C}^*$}

\newbox\gnBoxA
\newdimen\gnCornerHgt
\setbox\gnBoxA=\hbox{$\ulcorner$}
\global\gnCornerHgt=\ht\gnBoxA
\newdimen\gnArgHgt
\def\Quinequote #1{%
    \setbox\gnBoxA=\hbox{$#1$}%
    \gnArgHgt=\ht\gnBoxA%
    \ifnum     \gnArgHgt<\gnCornerHgt \gnArgHgt=0pt%
    \else \advance \gnArgHgt by -\gnCornerHgt%
    \fi \raise\gnArgHgt\hbox{$\ulcorner$} \box\gnBoxA %
    \raise\gnArgHgt\hbox{$\urcorner$}}

\DeclareMathOperator{\Incl}{Incl}
\DeclareMathOperator{\Cov}{Cov}
\DeclareMathOperator{\proj}{proj}
\DeclareMathOperator{\ct}{ct}

\newcommand{\bG}{\partial G}

\theoremstyle{plain}
\newtheorem{theorem}{Theorem}[section]
\newtheorem{lemma}[theorem]{Lemma}
\newtheorem{corollary}[theorem]{Corollary}
\newtheorem{proposition}[theorem]{Proposition}

\newtheorem{question}[theorem]{Question}
\newtheorem{fact}[theorem]{Fact}

\theoremstyle{definition}
\newtheorem{definition}[theorem]{Definition}
\newtheorem{example}[theorem]{Example}
\newtheorem{remark}[theorem]{Remark}

\numberwithin{equation}{section}

\begin{document}
\title[Computable presentations of hyperbolic group \cstar-algebras]{Computable presentations of reduced group \cstar-algebras of hyperbolic groups}

\author{Isaac Goldbring}
\address{University of California, Irvine}
\email{igoldbri@uci.edu}
\urladdr{https://www.math.uci.edu/~isaac}

\begin{abstract}
Answering a question of Thomas Sinclair, we show that the standard presentation of the reduced group \cstar-algebra $C^*_r(G)$ of an infinite hyperbolic group $G$ is computable.  The main ingredient in proving this result is showing that the Gromov boundary $\bG$ of $G$ admits a computably compact metric presentation for which the natural action of $G$ on $\bG$ is computable.   
\end{abstract}
\maketitle

%\tableofcontents

\section{Introduction}\label{sec:Intro}

In their paper \cite{canyou}, Fritz, Netzer, and Thom essentially began the computable structure theory for \cstar-algebras.  This work was later extended by Fox \cite{fox} and recently by the author and his collaborators in \cite{effnuc,UHF,AF}.

Roughly speaking, the question at hand is the following:  if $A$ is a separable \cstar-algebra, is the norm on $A$ computable?  In other words, is there an algorithm which, upon input some element $a$ of $A$, computes the norm of $a$?  Since algorithms are finitary objects, one needs to treat this question with care:  we must restrict the algorithm to only allow inputs from some countable dense sequence from $A$ and then we may only ask for better and better rational approximations to the norm of $a$.

In their paper, Fritz, Netzer, and Thom only considered the question for group \cstar-algebras (both universal and reduced).  In this case, a natural countable dense subset of the algebra arises, namely elements of the group ring with Gaussian rational coefficients.  They showed that for a fintely presented amenable group, computability of the group \cstar-algebra $\C^*_r(G)$ is equivalent to solvability of the word problem $G$.  (Fox \cite{fox} later established that the finite presentability is a red herring; the same result holds for any computable presentation of a group.)

In another direction, Fritz, Netzer, and Thom showed that the universal group \cstar-algebra $C^*(\mathbb{F}_2)$ associated to a free group is computable and asked whether or not the same was true for $C^*(\mathbb{F}_2\times \mathbb{F}_2)$, motivated by concerns connected to the now refuted Connes Embedding Problem.  In \cite{GS}, the author and Thomas Sinclair gave a negative answer to their question.

After finishing our paper, Sinclair asked the author about the status of the computability of the \emph{reduced} group \cstar-algebras $C^*_r(\mathbb{F}_2)$ and $C^*_r(\mathbb{F}_2\times \mathbb{F}_2)$.  Recently, the author came up with an argument (using an old technique of Choi \cite{choi} for embeddings the reduced group \cstar-algebra of the free group into the Cuntz algebra $\mathcal{O}_2$ for showing that these algebras do indeed have computable norm (see Appendix \ref{cuntz} for an outline of the argument), but this argument seemed incapable of generalization to a larger class of groups (such as the class of surface groups, again a question posed to the author by Sinclair).

It then occurred to the author that an alternate, \emph{dynamical} proof of the computability of $C^*_r(\mathbb{F}_2)$ might lend itself better to generalization.  To see why, it helps to revisit the argument for establishing computability of the \cstar-algebra for a computably presented amenable group.  In general, for a computably presented group $G$, the standard presentation of the universal group \cstar-algebra $C^*(G)$ is a \emph{right-c.e.} presentation, meaning that the norm is effectively approximable \emph{from above}, whereas the standard presentation of the reduced group \cstar-algebra $C^*_r(G)$ is a \emph{left-c.e.} presentation, meaning that the norm is effectively approximable \emph{from below}.  By amenability, the natural quotient map $C^*(G)\to C^*_r(G)$ is an isomorphism, allowing one to combine these two algorithms and effectively approximate the norm.

Now the free group $\mathbb{F}_2$ is not amenable, but it does \emph{act amenably} on its boundary $\partial \mathbb{F}_2$.  One can run a similar argument, this time with amenable actions and crossed products, to show that if a computably presented group acts computably and amenably on a computably compact metric space, then the standard presentation of the resulting crossed product \cstar-algebra is computable; the details of this argument are to appear in the forthcoming preprint \cite{effnuc}.  Since the reduced group \cstar-algebra of $G$ embeds computably in the aforementioned crossed product, computably of $C^*_r(G)$ follows.

One can show that the prefix metric on $\partial \mathbb{F}_2$ is computably compact and the action of $\mathbb{F}_2$ on $\partial \mathbb{F}_2$ is computable, leading to an alternate proof of the computability of $C^*_r(\mathbb{F}_2)$ using the route described in the previous paragraph; we say more about this in Section \ref{freesection} below.  The bulk of this paper is devoted to the following generalization of this fact:

\begin{theorem}
If $G$ is an infinite hyperbolic group, the boundary $\partial G$ admits a computably compact presentation for which the natural action of $G$ is computable.
\end{theorem}  

Since the action of a hyperbolic group on its boundary is known to be amenable by a result of Adams \cite{adams} (see also \cite[Theorem 5.3.15]{BO}), this establishes the main result of this paper:

\begin{theorem}
If $G$ is an infinite hyperbolic group, then the standard presentation of $C^*_r(G)$ is computable. 
\end{theorem}

It will be clear from the proof that in fact the computability of the presentation of $C^*_r(G)$ is uniform in a finite presentation of $G$ and a hyperbolicity constant $\delta$ for $G$.

The author would like to thank Thomas Sinclair for many useful discussions around this work.

\section*{AI Disclosure Statement}

The idea to show that a hyperbolic group acts computably on its boundary is due to the author.  At first, the author considered using the well-studied class of so-called \emph{visual} metrics on the boundary as candidates for computable metrics.  After failing to prove that these were computable, the author consulted Claude Opus 5 to see if there was an argument for computability of visual metrics.  Claude confirmed that this seemed unlikely and suggested a route to a computable presentation of the boundary via computable topology, which is the route taken here.  Claude was further used to help the author learn the relevant geometry of hyperbolic groups as well as to reconstruct the necessary computable topology more in line with the author's work, rather than the results in the literature stated in terms of type two effective theory.

The paper was written by the author and he takes full responsibility for its content.

\section{A sufficient criterion for computable presentability of the reduced group \cstar-algebra}

In this paper, we will need the notions of computable presentations of groups, metric spaces, and \cstar-algebras.  (Later, we will also introduce computable presentations of toplogical spaces.)  These are all special instances of the notion of a computable presentation of a \emph{metric structure} as expounded in the recent book \cite{EMST}.  For ease of the reader, we recall these definitions here in our three special cases.

\begin{definition}
A \emph{presented metric space} is a triple $X^\#:=(X,d,(x_n)_{n\in \N})$, consisting of a metric space $(X,d)$ and a countable dense sequence $(x_n)_{n\in \N}$ of $(X,d)$; in this case we also say that $X^\#$ is a \emph{presentation of $(X,d)$} and we refer to the points $x_n$ as the \emph{special points} of $X^\#$.  Such a presentation is said to be \emph{computable} if there is an algorithm such that, upon input $(m,n,k)\in \N^3$, returns a rational number $q>0$ such that $|d(x_m,x_n)-q|<2^{-k}$.  A computable presentation $X^\#$ of a compact metric space $(X,d)$ is \emph{computably compact} if there is an algorithm such that, upon input $k$, returns a finite sequence $i_1,\ldots,i_t\in \N$ such that $X$ is covered by the open balls around $x_{i_j}$ of radius $2^{-k}$.
\end{definition}

\begin{definition}
A \emph{presented \cstar-algebra} is a pair $(A,(a_n)_{n\in \N})$ consisting of a \cstar-algebra $A$ and a countable sequence $(a_n)_{n\in \N}$ from $A$ that generates $A$ as a \cstar-algebra; in this case, we say that $A^\#$ is a \emph{presentation of $A$} and we call the elements $a_n$ the \emph{special points} of $A^\#$.  We call the elements of $A$ of the form $p(a_{i_1},\ldots,a_{i_k})$, where $p$ is a $*$-polynomial with coefficients from $\mathbb{Q}(i)$, \emph{generated points} of $A^\#$; note that the generated points of $A^\#$ can be effectively enumerated and are dense in $A$.  The presentation $A^\#$ is \emph{computable} if there is an algorithm such that, upon input a generated point $a$ of $A^\#$ and rational $\epsilon>0$, returns a rational number $q>0$ such that $|\|a\|-q|<\epsilon$.  
\end{definition}

In regards to the previous definition, we make the following convention:  \textbf{throughout this paper, all \cstar-algebras are assumed to be unital and all presentations are assumed to include the identity as a special point.}

We will need the following fact due to McNicholl \cite{mcnicholl}:

\begin{fact}
Given a computably compact presentation $X^\#$ of a metric space $X$, there is a computable presentation $C(X^\#)$ of $C(X)$ that is \emph{evaluative over $X^\#$} in the sense that the evaluation map $C(X^\#)\times X^\#\to \mathbb{C}$ is computable,  meaning that there is an algorithm such that, upon input a generated point $f$ of $C(X^\#)$, $n\in \N$, and rational $\epsilon>0$, returns $q\in \mathbb{Q}(i)$ such that $|f(x_n)-q|<\epsilon$.  
\end{fact}

In the context of the previous fact, such an evaluative presentation is in fact unique (up to computable isomorphism).

The following definition is perhaps a bit nonstandard but fits in nicely with our current discussion (see also \cite[Definition 3.19]{fox}):

\begin{definition}
A \emph{presented group} is a pair $G^\#:=(G,(g_n)_{n\in \N})$, where $G$ is a group and $(g_n)_{n\in \N}$ is a countable sequence in $G$ generating $G$ as a group; in this case, we say that $G^\#$ is a \emph{presentation of $G$}.  We say that the presentation $G^\#$ is \emph{computable} if there is an algorithm such that, upon input a finite word in the sequence $(g_n)_{n\in \N}$, determines whether or not the word equals the identity in $G$.
\end{definition}

In the above definition, if one wishes to use a finite generating sequence for $G$, this can be accomplished by repeating the finite sequence infinitely often.  In this case, if $G$ is a finitely generated group with fintie generating set $S$, we may refer to the presentation of $G$ whose generating sequence is $S$ (with some given enumeration) repeated infinitely often as the \emph{presentation of $G$ induced by $S$}.  If $G^\#$ is the presentation of $G$ induced by $S$, then computability of $G^\#$ is equivalent to the solvability of the word problem for $G$.  Since solvability of the word problem is independent of the generating set, if $G$ is a finitely generated group with solvable word problem, then all such presentations of $G$ will be computable.

If $G^\#$ is a presentation of the group $G$, we may consider the presentation of the reduced group \cstar-algebra $C^*_r(G)$ whose collection of special points are the canonical unitaries $u_{g_n}$ associated to the distinguished points $g_n$ of the presentation $G^\#$.  We refer to this presentation as $C^*_r(G^\#)$ and call it the \emph{presentation of $C^*_r(G)$ associated to $G^\#$}.

Suppose that $G^\#$ is a presentation of the group $G$ and $X^\#$ is a presentation of the metric space $X$.  Suppose further that $G\curvearrowright^\alpha X$ is an action of $G$ on $X$ (by homeomorphisms).  We say that the action $G^\#\curvearrowright^\alpha X^\#$ is \emph{computable} if there is an algorithm such that, upon input a distinguished point $g$ of $G^\#$, a special point $x$ of $X^\#$, and rational $\epsilon>0$, returns a special point $y$ of $X^\#$ such that $d(\alpha(g,x),y)<\epsilon$.

The following criterion for computability of the reduced group \cstar-algebra can be extracted from \cite[Section 5]{effnuc}:

\begin{fact}\label{exacttheorem}
Suppose that $G\curvearrowright^\alpha X$ is an amenable action of the group $G$ on the compact space $X$.  Furthermore, assume that $G^\#$ is a computable presentation of $G$, $X^\#$ is a computably compact presentation of $X$, and $G^\#\curvearrowright^\alpha X^\#$ is a computable action.  Then the presentation $C^*_r(G^\#)$ is computable.
\end{fact}

To be more precise, in the context of the previous fact, it is shown in \cite[Section 5]{effnuc} that the induced presentation of the crossed product algebra $(C(X)\rtimes_\alpha G)^\#$ is computable.  On the other hand, it is clear from the definitions that the canonical embedding $C^*_r(G^\#)\hookrightarrow (C(X)\rtimes_\alpha G)^\#$ is computable, thus establishing the previous fact.

The main result of this paper is the following.  The relevant definitions will be provided in Section \ref{hyperbolic}.

\begin{theorem}\label{maintheorem}
Suppose that $G$ is an infinite hyperbolic group and that $G^\#$ is the presentation of $G$ associated to some finite generating set of $G$.  Then there is a computable presentation $\partial G^\#$ of the boundary $\partial G$ of $G$ such that the natural action of $G^\# \curvearrowright \partial G^\#$ is computable.  
\end{theorem} 

Adams \cite{adams} proved that the canonical action of $G$ on $\partial G$ is amenable.  As a result, we obtain:

\begin{corollary}
If $G$ is an infinite hyperbolic group with presentation $G^\#$ associated to some finite generating set of $G$, then the corresponding presentation $C^*_r(G^\#)$ is computable.
\end{corollary}

Fact \ref{exacttheorem} leads to the following natural question:

\begin{question}
Which exact groups admit computable presentations and computable, amenable actions on computably compact spaces?
\end{question}

\section{Motivation:  the case of the free group}\label{freesection}

Before diving in to the proof of Theorem \ref{maintheorem}, we explain the much easier case of the free group.  Consider the free group $\mathbb F_2$ freely generated by $a$ and $b$ and let $\partial \mathbb{F}_2$ denote its \emph{boundary} which, in this case, simply consists of all infinite reduced words in $a^{\pm 1},b^{\pm 1}$ (equivalently, the set of infinite geodesic words).  The boundary $\partial \mathbb{F}_2$ is naturally homeomorphic to the Cantor set.  We treat $\partial \mathbb{F}_2$ as a metric space using the prefix metric:  $d(\alpha,\beta)=2^{-k}$, where $k$ is minimal with the property $\alpha(k)\not=\beta(k)$.  Equipped with this metric, $\partial \mathbb{F}_2$ is in fact a computable metric space where the special points are the infinite reduced words that are eventually constant; this presentation is easily seen to be computably compact.  Finally, the natural action of $\mathbb{F}_2$ on $\partial \mathbb{F}_2$ given by concatenating and reducing is also easily seen to be a computable action.  Thus, letting $\mathbb{F}_2^\#$ denote the obvious standard presentation of $\mathbb{F}_2$ (associated to the finite generating set $\{a,b\}$), Fact \ref{exacttheorem} above allows us to conclude that $C^*_r(\mathbb{F}_2^\#)$ is computable.

A strategy similar to the above is available to us in the case of a hyperbolic group $G$.  The main catch is that the boundary $\bG$ of $G$ is not simply the set of infinite geodesic sequences; instead, many infinite geodesic sequences can represent the same element of $\bG$.  While it is still true that the set of infinite geodesic sequences is a computably compact space (see Proposition \ref{Piscompcomp} below), $G$ does not act on this space but rather its quotient $\bG$.  Although there is a very active study of so-called \emph{visual metrics} on $\bG$, none of them are obviously computable metrics in an appropriate sense.  As a result, we need to use the computably compact presentation on the space of infinite geodesic words to endow $\bG$ with a computably compact presentation for which the natural action of $G$ on $\bG$ is computable.  To do so requires a detour through some computable topology, which we now discuss.

\section{Some computable topology}

\begin{definition}\label{def:cts}
A \emph{computable topological space} is a pair $(X,(B_i)_{i\in\N})$, where $X$ is a $T_0$ topological space, $(B_i)_{i\in \N}$ is a numbered basis for its topology, and for which there is a c.e.\ relation $\Incl\subseteq\N^3$ which is both:
\begin{itemize}
\item \emph{sound}: $(i,j,k)\in\Incl$ implies $B_k\subseteq B_i\cap B_j$; and
\item \emph{rich}: whenever $x\in B_i\cap B_j$, there is $k$ with $x\in B_k$ and $(i,j,k)\in\Incl$.
\end{itemize}
\end{definition}

Soundness and richness are what allow one to compute with intersections: there is an algorithm such that, upon input $i,j\in \mathbb{N}$, enumerates a c.e. set $W_{i,j}\subseteq \N$ such that $B_i\cap B_j=\bigcup_{k\in W_{i,j}}B_k$. Taking $i=j$ shows that
\[
\{(i,k) : (k,k,i)\in\Incl\}
\]
is a c.e.\ subrelation of $\{(i,k):B_i\subseteq B_k\}$.

In what follows, we fix a computable topological space $(X,(B_i)_{i\in \N})$ with c.e. inclusion relation $\Incl$.  We often abuse notation and refer to a computable topological space simply by its underlying topological space, suppressing mention of the specified numbered basis.

An open set $U\subseteq X$ is \emph{effectively open} if $U=\bigcup_{i\in W}B_i$ for some c.e.\ $W\subseteq \N$; such a $W$ is called a \emph{name} of $U$. A set is \emph{effectively closed} if its complement is effectively open.

A \emph{name of a point} $x$ is an enumeration of $\{i\in \N:x\in B_i\}$; $x$ is \emph{computable} if it has a computable name. A map $f:X\to Y$ is \emph{computable} if there is an algorithm such that, upon input a name of $x$, returns a name of $f(x)$; equivalently, if $f^{-1}(B^Y_j)$ is effectively open, uniformly in $j$. Note that computable maps are continuous.

\begin{example}\label{msexample}
Suppose $X^\#=(X,d,(x_n)_{n\in \mathbb{N}})$ is a computable metric space.  From it one can obtain a computable topological space $(X,(B_i)_{i\in \mathbb{N}})$ whose numbered basis is an effective enumeration of the open balls with centers coming from the sequence $(x_n)_{n\in \mathbb{N}}$ and rational radii.  In what follows, we will always treat computable metric spaces as computable topological spaces in this way.  Note that the notions defined above (effectively open, effectively closed, computable point, computable map) agree with the corresponding notions from computable metric space theory, whence no confusion can arise.
\end{example}

In what follows, we let $\langle \cdots\rangle$ denote any effective coding function $\N^{<\N}\to \N$.  We also use the easy fact that the category of computable topological spaces (with computable maps as morphisms) admits products.

\begin{definition}\leavevmode
Let $(X,(B_i)_{i\in \mathbb{N}})$ be a computable topological space.  We say that $X$ is
\begin{itemize}
\item \emph{computably compact} if it is compact and
\[
\Cov:=\{\langle i_1,\dots,i_n\rangle : X=B_{i_1}\cup\cdots\cup B_{i_n}\}
\]
is c.e. 
\item \emph{computably Hausdorff} if the complement of the diagonal is effectively open in $X\times X$, that is, there is a c.e.\ set $W$ of pairs $(i,j)$ such that $$\{(x,y)\in X\times X:x\ne y\}=\bigcup_{(i,j)\in W}B_i\times B_j.$$
\item \emph{computably overt} if $\{i:B_i\ne\emptyset\}$ is c.e.
\end{itemize}
\end{definition}

Note the following easy observations:
\begin{itemize}
\item $X$ is computably compact if and only if from an enumeration of a basic cover, one can compute a finite subcover.
\item $X$ is computably overt if and only if:  given a name of an open $U$, the statement ``$U\ne\emptyset$'' is semi-decidable.
\item Any computable metric space, viewed as a computable topological space as in Example \ref{msexample}, is automatically computably Hausdorff and computably overt.  Moreover, it is computably compact qua computable metric space if and only if it is computably compact qua computable topological space.  Similar remarks can be made about computable points and computable maps.
\end{itemize}

\begin{lemma}\label{lem:proj}
Suppose that $X$ and $Y$ are computable topological spaces with $X$ computably compact.  Further suppose that $Z\subseteq X\times Y$ is effectively closed.  Then $\proj_Y(Z)$ is effectively closed, uniformly in all of the relevant data.
\end{lemma}

\begin{proof}
Note that $y\notin\proj_Y(Z)$ if and only if $X\times\{y\}\subseteq(X\times Y)\setminus Z$. By compactness of $X$ this holds if and only if finitely many basic boxes cover $X\times\{y\}$ inside the complement, and the search for such a finite family is an enumeration.
\end{proof}

The following is probably well known; we include a proof for the sake of completeness.

\begin{lemma}\label{lem:image}
Suppose that $X$ and $Y$ are computable topological spaces and $f:X\to Y$ is computable and surjective.
\begin{enumerate}
    \item If $X$ is computably compact, then so is $Y$.
    \item If $X$ is computably overt, then so is $Y$.
\end{enumerate}  
\end{lemma}

\begin{proof}
For (1), by surjectivity, $Y=\bigcup_t B^Y_{j_t}$ if and only if $X=\bigcup_t f^{-1}(B^Y_{j_t})$, and each $f^{-1}(B^Y_{j_t})$ is effectively open uniformly because $f$ is computable; now apply computable compactness of $X$. Similarly, for (2), $B^Y_j\ne\emptyset$ if and only if $f^{-1}(B^Y_j)\ne\emptyset$, which is semi-decidable by the computable overtness of $X$.
\end{proof}

The next two results are not needed until Appendix \ref{metrizationappendix}.

\begin{lemma}\label{lem:graph}
Let $f:X\to Y$ be computable with $Y$ computably Hausdorff. Then the graph of $f$ is effectively closed in $X\times Y$. If moreover $X$ is computably compact, then $f$ maps effectively closed sets to effectively closed sets, and uniformly so.
\end{lemma}

\begin{proof}
With $W$ the Hausdorff witness set for $Y$, we have
\[
\{(x,y)\in X\times Y : y\ne f(x)\}=\bigcup_{(j,k)\in W} f^{-1}(B^Y_j)\times B^Y_k,
\]
which is effectively open. For the second claim, $f(C)=\proj_Y\big(\mathrm{graph}(f)\cap(C\times Y)\big)$ and Lemma~\ref{lem:proj} applies.
\end{proof}

\begin{corollary}\label{cor:inverse}
A computable bijection from a computably compact space to a computably Hausdorff space has computable inverse.
\end{corollary}

\begin{proof}
For basic open $B_i\subseteq X$, bijectivity gives $f(B_i)=Y\setminus f(X\setminus B_i)$, and $f(X\setminus B_i)$ is effectively closed by Lemma~\ref{lem:graph}. So $(f^{-1})^{-1}(B_i)=f(B_i)$ is effectively open uniformly in $i$.
\end{proof}

Let $X$ be a computable topological space and $\sim$ an equivalence relation on $X$ which is effectively closed as a subset of $X\times X$. Set $Y=X/\!\sim$, viewed as a topological space equipped with the quotient topology, and let $\pi:X\to Y$ denote the canonical projection map.  For $C\subseteq X$, set $$R[C]:=\{x\in X \ : \ \text{ there is } c\in C \text{ such that}\ c\sim x\},$$ the saturation of $C$. For open $U\subseteq X$, set
\[
U^\ast:=\{\pi(x)\in Y : R[x]\subseteq U\}=Y\setminus\pi(X\setminus U),
\]
the largest subset of $Y$ whose $\pi$-preimage lies in $U$. The star operation is monotone, preserves finite intersections, and satisfies $V=(\pi^{-1}V)^\ast$ for every open $V\subseteq Y$, whence every open set is of the form $U^*$ for an open subset $U$ of $X$.

% This basis must be indexed by finite unions of basic sets of $X$, not by single ones. A fibre $R[x]$ is compact but in general lies inside no single $B_i$, and $U^\ast$ is empty unless $U$ engulfs whole fibres. 

For finite $A\subseteq\N$, set
\[
B_A:=\bigcup_{i\in A}B_i,\qquad C_A:=(B_A)^\ast .
\]

\begin{lemma}\label{lem:quotient}
Let $X$ be computably compact and $\sim$ an effectively closed equivalence relation on $X$ for which $Y:=X/\sim$ is $T_0$. Then the following hold:
\begin{enumerate}
    \item $(C_A)_{A\ \mathrm{finite}}$ is a numbered basis making $Y$ a computable topological space.
    \item The canonical projection $\pi:X\to Y$ is computable.
    \item $Y$ is computably compact.
    \item If $X$ is computably overt, so is $Y$.
\end{enumerate}
\end{lemma}

\begin{proof}
Consider $\pi(x)\in V$ with $V\subseteq Y$ open. Then $\pi^{-1}(V)$ is a saturated open set containing the compact set $R[x]$, whence finitely many basic sets contained in $\pi^{-1}(V)$ cover $R[x]$; let $A$ index them. Then $R[x]\subseteq B_A\subseteq\pi^{-1}(V)$, so $\pi(x)\in C_A\subseteq V$.

Next note that $\pi^{-1}(C_A)=X\setminus R[X\setminus B_A]$. For effectively closed $C$, we have
\[
R[C]=\proj_2\big({\sim}\cap(C\times X)\big),
\]
which is effectively closed uniformly by Lemma~\ref{lem:proj}.  Hence $\pi^{-1}(C_A)$ is effectively open uniformly in $A$, which is computability of $\pi$.  Note in particular that this shows that each $C_A$ is open.

Declare $(C_A,C_B,C_C)\in\Incl_Y$ when $\Incl_X$ certifies $B_C\subseteq B_A\cap B_B$, that is, when, for each $k\in C$, there are $i\in A$ and $j\in B$ such that $(i,j,k)\in \Incl_X$.  Note that $\Incl_Y$ is indeed c.e. Soundness holds because $\ast$ is monotone and meet-preserving. Richness holds because the basis argument above, applied to $V=C_A\cap C_B$, produces exactly such a $C$.  

At this point, we have established (1) and (2).  (3) and (4) follow from Lemma~\ref{lem:image}.
\end{proof}

Crucial for our purposes is the following result:

\begin{theorem}\label{thm:metrize}
A computably compact, computably Hausdorff, computably overt computable topological space admits a computable metric space presentation computably equivalent to the given one.
\end{theorem}

Let us be more precise about the statement of the previous theorem:  Suppose that $(X,(B_i)_{i\in \mathbb{N}})$ is a computably compact, computably Hausdorff, computably overt computable topological space.  Let $X^\#$ be the computable metric space guaranteed to exist by the theorem.  Let $(X,(C_i)_{i\in \mathbb{N}})$ be the computable topological space associated to $X^\#$ as in Example \ref{msexample}.  Then the conclusion of the theorem is that the identity map $X\to X$ is a computable map from $(X,(B_i)_{i\in \mathbb{N}})$ to $(X,(C_i)_{i\in \mathbb{N}})$ and vice-versa.  Note this automatically implies that $X^\#$ is a computably compact metric space presentation.

A version of Theorem \ref{thm:metrize} has appeared before in the literature in the context of type two effective theory \cite{GW}.  We prefer to give a proof that is more in line with our approach to computable analysis.  We give such a self-contained proof of Theorem \ref{thm:metrize} in Appendix \ref{metrizationappendix}.

Putting everything from this section together, we arrive at the following:

\begin{corollary}\label{cor:quotmetrize}
Let $X$ be a computably compact and computably overt computable topological space.  Let $\sim$ be an effectively closed equivalence relation. Equip the quotient space $Y = X/\sim$ with the above effective quotient presentation constructed from finite unions of basic opens. If $Y$ is computably Hausdorff, then $Y$ has a computably equivalent computably compact metric space presentation, and the quotient map $q:X\to Y$ is computable.
\end{corollary}

\section{Computable presentation of the boundary of a hyperbolic group}\label{hyperbolic}

We begin by recalling some basic facts about hyperbolic groups to be used throughout this section.    

Fix a group $G$ with finite generating set $S$; for simplicity, we always assume that our generating sets are closed under inverse and do not contain the identity.  We set $\Gamma:=\Gamma(G,S)$ for the Cayley graph of $G$ with respect to $S$, which we view as a proper, geodesic metric space on which $G$ acts by isometries on the left. For $x\in G$, set $|x|:=d(e,x)$.  Set $S_n:=\{x\in G : |x|=n\}$ and $B(x,r):=\{y\in G:d(x,y)\le r\}$.

A geodesic triangle in $\Gamma$ is \emph{$\delta$-thin} if each side lies in the closed $\delta$-neighborhood of the union of the other two. $\Gamma$ is said to be \emph{$\delta$-hyperbolic} if all geodesic triangles in $\Gamma$ are $\delta$-thin and $G$ is said to be $\delta$-hyperbolic if $\Gamma(G,S)$ is $\delta$-hyperbolic for some finite generating set $S$.  We also say that $\Gamma$ is hyperbolic if it is $\delta$-hyperbolic for some $\delta>0$ and we say that $G$ is hyperbolic if $\Gamma(G,S)$ is hyperbolic for some generating set $S$.  Unlike $\delta$-hyperbolicity for a given $\delta$, hyperbolicity is a quasi-isometry invariant; since all Cayley graphs of a finitely generated group are quasi-isometric, we have that $G$ is hyperbolic if and only if $\Gamma(G,S)$ is hyperbolic for all generating sets $S$.

For the rest of this section, we suppose that $S$ is a finite generating set witnessing that $\Gamma:=\Gamma(G,S)$ is $\delta$-hyperbolic, with $\delta\in \N$.  (Since $\delta$-hyperbolic implies $\delta'$-hyperbolic for any $\delta'>\delta$, this is no loss of generality.)  We equip $G$ with the presentation $G^\#$ associated to $S$ as discussed in Section 2.

It is well-known that $G$ admits a so-called \emph{Dehn presentation} $G=\langle S|R\rangle$, which then allows one to use \emph{Dehn's algorithm} to solves the word problem.  In what follows, every conclusion is uniform in $\delta$ and $\langle S|R\rangle$.

For each $n\in \N$, the labelled ball $B_n:=B(e,n)$, as a finite edge-labelled graph together with the identification of which words name which vertices, is computable from $n$, and the predicate ``$w$ is a geodesic word'' is decidable.

The \emph{geodesic cone type} of $x\in G$ is
\[
\ct(x):=\{w\in S^* : |xw|=|x|+|w|\}.
\]
Cannon \cite{cannon} proved that in a $\delta$-hyperbolic group, there are only finitely many cone types. In fact, Cannon’s proof shows effectively that cone types are determined by their $k$-types for an integer $k$ computable from a hyperbolicity constant for $G$.  Let $W$ be the  finite state automaton with states the cone types, start state $\ct(e)$, and a transition $\ct(x)\xrightarrow{\,s\,}\ct(xs)$ (for $s\in S$) whenever $|xs|=|x|+1$. Every state is accepting. The finite paths of $W$ from the start state are exactly the geodesic words.

Delete from $W$ every state from which no infinite path emanates (equivalently, from which no cycle is reachable).  Write $W^{\mathrm{tr}}$ for the resulting \emph{trimmed} automaton and let $P$ denote the set of infinite paths of $W^{\mathrm{tr}}$ starting from the start state.  Using other language, $P$ is the set of geodesic rays begining at $e$; here, a \emph{geodesic ray} in $\Gamma$ is an isometric embedding $\alpha:[0,\infty)\to \Gamma$ and we say that $\alpha$ \emph{begins at} $\alpha(0)$ or that $\alpha$ is \emph{from} $\alpha(0)$. Note that $P\ne\emptyset$ if and only if $G$ is infinite. Set
\[
P_n := \{x\in S_n : x \text{ lies on some geodesic ray from } e\},
\]
a decidable subset of $S_n$ (being the level-$n$ vertices surviving the trimming procedure).  Since finite labelled balls and word lengths are computable from our Dehn presentation, the cone types and their transition table can be computed. Hence $W$ and $W^{\operatorname{tr}}$ are effectively constructible from $\delta$ and $\langle S|R\rangle$.

$P$ carries the prefix metric $d_P(\alpha,\beta):=2^{-\min\{n:\alpha(n)\ne\beta(n)\}}$.  We can view $P$ as a presented metric space whose special points are the set of eventually periodic elements (enumerated in some fixed effective way).  These are indeed dense for the following reason:  given a geodesic ray $\alpha\in P$ and $n\in \mathbb{N}$, the automaton must visit the same state twice after landing at $\ct(\alpha(n))$, say at $r>s\geq n$; if we let $w=\alpha(1)\cdots\alpha(n)$, $y=\alpha(n+1)\cdots \alpha(s)$, and $v=\alpha(s+1)\cdots\alpha(r)$, then $wyv^\infty$ is an eventually periodic word within distance $2^{-n}$ of $\alpha$.  (Here, $wyv^\infty$ stands for the geodesic ray $wyvvv\cdots\in P$.)  We refer to this presentation of $P$ as the \emph{standard presentation}.

\begin{proposition}\label{Piscompcomp}
The standard presentation of $P$ is computably compact.
\end{proposition} 

\begin{proof}
That the presentation is computable is immediate from the definition.  To verify computable compactness, fix $n$ and for each $x\in P_n$, search for geodesic words extending $x$ that eventually reach the same state twice and use the same argument as above to compute an eventually periodic word $\alpha_x$ extending $x$.  Then for any infinite word $\alpha$, setting $x(\alpha):=\alpha(1)\cdots\alpha(n)\in P_n$, we have $d(\alpha,\alpha_{x(\alpha)})\leq 2^{-n}$.  It follows that the open balls with centers $\alpha_x$, for $x\in P_n$, and radius $2^{-n}$, form a finite open cover of $P$. 
\end{proof}

Two geodesic rays in $\Gamma$ are said to be \emph{asymptotic} if their images have finite Hausdorff distance.  We let $\sim$ denote the equivalence relation on $P$ given by $\alpha\sim \beta$ if and only if they are asymptotic.  This is clearly a closed equivalence relation on $P$.  In hyperbolic groups, there is a very useful reformulation of this equivalence relation defined in terms of so-called \emph{fellow-traveling}:

\begin{fact}\label{lem:ft}
For $\alpha,\beta\in P$, we have that $\alpha\sim \beta$ if and only if $d(\alpha(t),\beta(t))\le 2\delta$ for all $t\ge 0$.
\end{fact}

An important consequence of fellow-traveling is the following:

\begin{lemma}
$\sim$ is an effectively closed equivalence relation on $P$.
\end{lemma}

\begin{proof}
Given Fact \ref{lem:ft}, the lemma follows from the equality
$$\{(\alpha,\beta)\in P\times P \ : \ \alpha\not\sim\beta\}=\bigcup_{t\in \N}\{(\alpha,\beta)\in P\times P \ : \ d(\alpha(t),\beta(t))>2\delta\}.\qedhere$$
\end{proof}

We set $\bG:=P/\sim$, topologized with the quotient topology, and let $\pi:P\to\bG$ be the canonical quotient map.  We refer to $\bG$ as the \emph{boundary of $G$}.

From Lemma \ref{lem:quotient}, we immediately obtain:

\begin{proposition}
Equip $\bG$ with the computable topology it obtains from $P$ and $\sim$ as in Lemma \ref{lem:quotient}.  Then $\bG$ is computably compact and computably overt and the quotient map $\pi:P\to \bG$ is computable.
\end{proposition}

In order to show that $\bG$ is computably Hausdorff, and thus that $\bG$ is effectively metrizable, it will behoove us to make explicit the numbered basis of $\bG$.

For $\xi\in\bG$, set
\[
T_n(\xi) := \{\alpha(n) : \alpha\in P \text{ and }\pi(\alpha)=\xi\}\ \subseteq P_n,
\]
a nonempty subset of a ball of radius $2\delta$ by Fact~\ref{lem:ft}. For $n\in\N$ and $F\subseteq P_n$, define
$$
B_{n,F} := \{\xi\in\bG : T_n(\xi)\subseteq F\},$$
indexed by the decidable set of pairs $(n,F)$.

\begin{lemma}\label{lem:basis}
The following facts hold:
\begin{enumerate}
    \item $B_{n,F}=(U_F)^\ast$, where $U_F:=\{\alpha\in P : \alpha(n)\in F\}$.
    \item Each $B_{n,F}$ is open.
    \item The family $(B_{n,F})$ is cofinal among the sets $C_A$ of Lemma~\ref{lem:quotient}, and is therefore a basis for $\bG$.  
    \item If $\alpha\in P$ and $\pi(\alpha)=\xi$, then setting $F:=B(\alpha(n),2\delta)\cap P_n$, we have $\xi\in B_{n,F}$.
\end{enumerate}
\end{lemma}

\begin{proof}
(1) follows from the fact that $$(U_F)^\ast=\{\xi\in \bG : \pi^{-1}(\xi)\subseteq U_F\}=\{\xi\in \bG : T_n(\xi)\subseteq F\}=B_{n,F}.$$  (2) follows from (1) and Lemma~\ref{lem:quotient}.  It is clear that the $U_F$ are cofinal among finite unions of basic sets of $P$, and thus the $B_{n,F}$ are cofinal among the sets $C_A$; (3) now follows again from Lemma~\ref{lem:quotient}.  For (4): $T_n(\xi)\subseteq B(\alpha(n),2\delta)$ by Fact~\ref{lem:ft}, whence $\xi\in B_{n,F}$. 
\end{proof}

We are now ready to prove that $\bG$ is computably Hausdorff:

\begin{proposition}\label{prop:hausdorff}
Let $W$ be the decidable set of pairs of pairs $((n,F),(n,F'))$, with $F,F'\subseteq P_n$ and $F\cap F'=\emptyset$.  Then 
$$\{(\xi,\eta)\in \bG\times \bG \ : \ \xi\not=\eta\}=\bigcup_{((n,F),(n,F'))\in W} B_{n,F}\times B_{n,F'}.$$
Consequently, $\bG$ is computably Hausdorff.
\end{proposition}

\begin{proof}
The inclusion ``$\supseteq$'' is obvious.  We now prove the inclusion ``$\subseteq$.''  Towards that end, consider $\xi\ne\eta$ and fix $\alpha,\beta\in P$ with $\pi(\alpha)=\xi$ and $\pi(\beta)=\eta$. Since $\alpha\not\sim \beta$, we may pick $n$ with $d(\alpha(n),\beta(n))>4\delta$ and set $F:=B(\alpha(n),2\delta)\cap P_n$, $F':=B(\beta(n),2\delta)\cap P_n$. Then $\xi\in B_{n,F}$ and $\eta\in B_{n,F'}$ by Lemma~\ref{lem:basis}, and $F\cap F'=\emptyset$ by the triangle inequality.
\end{proof}

Combining everything together, we have:

\begin{theorem}
There is a computably compact presentation $\bG^\#$ of $\bG$ as a metric space so that the canonical quotient map $\pi:P\to \bG$ is a computable map. 
\end{theorem}

\begin{remark}
It can be shown that the special points of the computable presentation of $\bG^\#$ can be taken to be the images under $\pi$ of the eventually periodic elements of $P$.  Indeed, these points are uniformly computable and dense since these same assertions hold of the eventually periodic points in $P$.
\end{remark}

It remains to show that the action of $G$ on $\bG$ is computable.  (Note that the computability of the action of $G$ on $\bG$ does not depend on whether or not we are thinking of $\bG$ as a computable topological space of a computable metric space.)  To accomplish that, it will first behoove us to recall the following well-known fact:

\begin{fact}\label{basepointfact}
Suppose that $\gamma$ is a geodesic ray in $\Gamma$ starting at $g$.  Then there is $\alpha\in P$ in $\Gamma$ starting at $e$ such that the Hausdorff distance of $\alpha$ and $\gamma$ is at most $|g|+2\delta$. 
\end{fact}

\begin{proof}[Proof sketch]
For each $n$, choose a geodesic word $w_n$ from $e$ to $\gamma(n)$.  By a K\"onig's lemma argument, we may choose a sequence $(u_k)_{k\in \mathbb{N}}$ such that each $u_k$ is a prefix of $w_n$ for infinitely many $n$, each $u_k$ is a prefix of $u_{k+1}$, and the lengths of $u_k$ tend to $\infty$.  Let $\alpha\in P$ denote the geodesic ray given by the limit of this sequence.  It is a calculation using thinness to show the claimed upper bound on the Hausdorff distance between $\alpha$ and $\gamma$.     
\end{proof}

Now suppose that $\xi\in \bG$ and consider $\alpha\in P$ with $\pi(\alpha)=\xi$.  Then the map $t\mapsto g\alpha(t)$ is also a geodesic ray in $\Gamma$, whhich we call $g\alpha$, which starts at $g$.  Using the above fact, one may find $\alpha'\in P$ within finite Hausdorff distance of the ray $g\alpha$.  One can show that replacing $\alpha$ and/or $\alpha'$ with equivalent geodesic rays leads to equivalent geodesic rays, whence we may define $g\xi:=\pi(\alpha')$ for any $\alpha'$ as above.

The following two results are probably well-known, but we include proofs for completeness.

\begin{lemma}\label{lem:basepoint}
Let $\xi\in\bG$ and suppose $\alpha\in P$ is such that $\pi(\alpha)=\xi$. Suppose also that $\gamma$ is a geodesic ray starting at $g\in G$ and staying within a finite Hausdorff distance of $\alpha$. Then for every $t>|g|+\delta$, there is $s$ with $|s-t|\le|g|+\delta$ and $d(\gamma(t),\alpha(s))\le 3\delta$.
\end{lemma}

\begin{proof}
Fix $n\geq t$ and consider the triangle $e,g,\gamma(n)$ with sides a geodesic $[e,g]$, $\gamma|_{[0,n]}$, and a geodesic $w_n$ connecting $e$ and $\gamma(n)$. For $t>|g|+\delta$, we have $d(\gamma(t),[e,g])\ge t-|g|>\delta$, so $\gamma(t)$ is within $\delta$ of $w_n$.  Letting $\alpha'\in P$ be constructed from the sequence $(w_n)_{n\in \N}$ as in the proof of Fact \ref{basepointfact}, we have that $d(\gamma(t),\alpha'(s))\le\delta$ for some $s$ which necessarily satisfies $|s-t|\leq |g|+\delta$.  By Fact~\ref{lem:ft}, $d(\alpha'(s),\alpha(s))\le2\delta$, finishing the proof.
\end{proof}

\begin{lemma}\label{lem:shrink}
Fix $D\ge0$ and $\zeta\in\bG$. Let $N\subseteq\N$ be infinite and suppose that, for each $n\in N$, we are given $F_n\subseteq P_n$ with
\[
T_n(\zeta)\subseteq F_n \qquad\text{and}\qquad \diam F_n\le D .
\]
Then $\{B_{n,F_n}\}_{n\in N}$ is a neighborhood basis at $\zeta$ consisting of open sets.
\end{lemma}

\begin{proof}
Each $B_{n,F_n}$ is open and contains $\zeta$ by Lemma~\ref{lem:basis}.  Suppose, towards a contradiction, that the family is not a neighborhood basis, that is, some open $U\ni\zeta$ contains no $B_{n,F_n}$.  As a result, for each $n\in N$, there is $\eta_n\in B_{n,F_n}\setminus U$. For each $n\in N$, choose $\beta_n\in P$ such that $\pi(\beta_n)=\eta_n$.  Choose also $\alpha\in P$ such that $\pi(\alpha)=\zeta$.  Note that $d(\alpha(n),\beta_n(n))\le D$.

Fix $t$ and take $n\in N$ with $n>t+D+\delta$. Consider the geodesic triangle with vertices $e$, $\alpha(n)$, $\beta_n(n)$, two of whose sides are $\alpha|_{[0,n]}$ and $\beta_n|_{[0,n]}$. Every point of the third side has distance at least $n-D$ from $e$, so $d\big(\beta_n(t),[\alpha(n),\beta_n(n)]\big)\ge n-D-t>\delta$. Thinness therefore places $\beta_n(t)$ within $\delta$ of $\alpha([0,n])$, say $d(\beta_n(t),\alpha(s))\le\delta$; then $|t-s|=\big||\beta_n(t)|-|\alpha(s)|\big|\le\delta$ and so $d(\beta_n(t),\alpha(t))\le2\delta$.

By compactness of $P$, we may pass to a subsequence $(\beta_{n_k})_{k\in \N}$ with $\beta_{n_k}\to\beta\in P$ in the prefix metric. For each fixed $t$, we have $\beta_{n_k}(t)=\beta(t)$ for large $k$, and the previous paragraph applies once $n_k>t+D+\delta$, so $d(\beta(t),\alpha(t))\le2\delta$ for every $t$. Hence $\beta\sim\alpha$.  Since $\pi$ is continuous, we have $\eta_{n_k}=\pi(\beta_{n_k})\to\pi(\beta)=\zeta$.  Since $\bG\setminus U$ is closed, we have $\zeta\notin U$, which is a contradiction.
\end{proof}

The final ingredient (and perhaps the least trivial) is the following:

\begin{theorem}\label{inclusiondecidable}
The set of pairs of pairs $$\{((n,F),(n',F')) \ : \  B_{(n,F)}\subseteq B_{(n',F')}\}$$ is decidable.
\end{theorem}

For the moment we assume that Theorem \ref{inclusiondecidable} holds and show how it yields the computability of the action:

\begin{proposition}\label{prop:action}
The action of $G^\#$ on $\bG^\#$ is computable.
\end{proposition}

\begin{proof}
 We need to describe an algorithm which, given a word representing $g$ and a name of $\xi\in\bG$, enumerates a name of $g\xi$.  Take $\alpha,\alpha'\in P$ such that $\pi(\alpha)=\xi$ and $\pi(\alpha')=g\xi$, noting that $\alpha'$ is within finite Hausdorff distance of the geodesic $t\mapsto g\alpha(t)$ beginning at $g$.  Lemma~\ref{lem:basepoint} applied with $\gamma=g\alpha$ gives, for $n>|g|+\delta$, an $s$ with $|s-n|\le|g|+\delta$ and $d(g\alpha(n),\alpha'(s))\le3\delta$, from which it follows that $d(g\alpha(n),\alpha'(n))\le4\delta+|g|$. As $\alpha'$ was arbitrary, we have
\begin{equation}\label{eq:transport}
T_n(g\xi)\ \subseteq\ B\big(g\alpha(n),\,4\delta+|g|\big)\cap P_n\qquad(n>|g|+\delta).
\end{equation}

Now, given a name of $\xi$, enumerate the indices $(n,F)$ with $\xi\in B_{n,F}$, and retain those with $n>|g|+\delta$ and $\diam F\le4\delta$, a decidable procedure. By the proof of Lemma~\ref{lem:basis}, the index $\big(n,B(\alpha(n),2\delta)\cap P_n\big)$ is retained for \emph{every} such $n$, so the retained indices occur at an infinite set $N$ of levels. For each retained $(n,F)$, pick any $x\in F$; since $\alpha(n)\in T_n(\xi)\subseteq F$ and $\diam F\le4\delta$, we have $d(x,\alpha(n))\le4\delta$, hence $d(gx,g\alpha(n))\le4\delta$, and \eqref{eq:transport} gives
\[
T_n(g\xi)\ \subseteq\ F'_n:=B\big(gx,\,8\delta+|g|\big)\cap P_n .
\]
Output the index $(n,F'_n)$; thus $g\xi\in B_{n,F'_n}$.

We have $\diam F'_n\le D:=2(8\delta+|g|)$, a bound independent of $n$, and $T_n(g\xi)\subseteq F'_n$ for all $n\in N$. Lemma~\ref{lem:shrink} therefore makes $\{B_{n,F'_n}\}_{n\in N}$ a neighborhood basis at $g\xi$ of open sets.  

A name of $g\xi$ must enumerate every $j$ with $g\xi\in B_j$. Run the outputs $B_{n,F'_n}$ against all indices $j$, enumerating $j$ as soon as $B_{n,F'_n}\subseteq B_j$ is certified; since the inclusion relation is decidable by Theorem \ref{inclusiondecidable} (c.e. would suffice here), this is indeed possible.  This produces the desired name of $g\xi$.
\end{proof}

It remains to prove Theorem \ref{inclusiondecidable}.  Fix $n\in\N$ and $F\subseteq P_n$, and set
\[
E_{n,F}:=\{\alpha\in P : \text{ there is }\beta\in P \text{ with } \beta\sim\alpha \text{ and } \beta(n)\notin F\}.
\]
It is clear that $\pi^{-1}(B_{n,F})=P\setminus E_{n,F}$.  By surjectivity of $\pi$, we have that $B_{n,F}\subseteq B_{n',F'}$ if and only if $E_{n',F'}\subseteq E_{n,F}$.

We now build a deterministic safety automaton $\mathcal D_{n,F}$, reading $\alpha$, whose survival detects a partner $\beta$ as in the definition of $E_{n,F}$; in other words, the $\omega$-language determined by $D_{n,F}$ is precisely $E_{n,F}$.

\

\noindent \textbf{Phase 1, steps $t< n$}: States are pairs $(x,Q)$ with $x\in S_t$, $t\le n$, and $Q\subseteq B(x,2\delta)\cap S_t$.  (In what follows, the second element of the state pair is often referred to as the \emph{partner set}). The start state is $(e,\{e\})$. On a letter $s$ with $|xs|=|x|+1$, the transition function is defined as
\[
\tau\big((x,Q),s\big)=\Big(xs,\ \big\{ys' \ :\ y\in Q,\ s'\in S,\ |ys'|=|y|+1,\ d(xs,ys')\le2\delta\big\}\Big),
\]
and is undefined if $|xs|\ne|x|+1$.

\

\noindent \textbf{Step $n$}: On arriving at a state $(x,Q)$ with $|x|=n$, replace $Q$ by $Q\setminus F$.

\

\noindent \textbf{Phase 2, steps $t>n$}: States are pairs $(c,R)$ where $c$ is a cone type and $R$ is a set of pairs $(c',u)$ with $c'$ a cone type and $u\in B(e,2\delta)$. The transition from phase 1 at time $n$ sends $(x,Q)$ to $\big(\ct(x),\{(\ct(y),x^{-1}y):y\in Q\}\big)$. Thereafter, on a letter $s\in c$,
\[
\tau\big((c,R),s\big)=\Big(c\cdot s,\ \big\{(c'\cdot s',\ s^{-1}us') \ :\ (c',u)\in R,\ s'\in c',\ |s^{-1}us'|\le2\delta\big\}\Big),
\]
undefined if $s\notin c$. This is well defined because whether $s'$ extends a partner geodesically depends only on its cone type $c'$, the new cone type is $c'\cdot s'$, and the new offset is
\[
\alpha(t+1)^{-1}\beta(t+1)=(\alpha(t)s)^{-1}(\beta(t)s')=s^{-1}\big(\alpha(t)^{-1}\beta(t)\big)s'=s^{-1}us' .
\]

In both phases the second component may become empty; in this case, make $\emptyset$ absorbing, so that once the partner set dies it stays dead while the machine continues to consume letters. The state set is indeed finite: in phase 1 there are at most $|B(e,n)|\cdot2^{|B(e,2\delta)|}$ states, and in phase 2 there are at most $K\cdot2^{K|B(e,2\delta)|}$ where $K$ is the number of cone types.

\begin{proposition}\label{lem:key}
   Let $\alpha\in P$. Then $\alpha\in E_{n,F}$ if and only if the partner component of the run of $\mathcal D_{n,F}$ on $\alpha$ is nonempty at every step. 
\end{proposition}

\begin{proof}
Let $T$ be the set of finite words $v$ over $S$, say of length $t$, such that (i) $v$ is geodesic, (ii) $d(\alpha(s),v(s))\le2\delta$ for all $s\le t$, and (iii) if $t\ge n$, then $v(n)\notin F$. Note that $T$ is a finitely branching tree under prefixes.

An induction on $t$ shows that the partner component of the run at time $t$ is exactly the set of endpoints $\{\bar v : v\in T,\ |v|=t\}$ in phase 1 (where $\bar v$ denotes the element of $G$ represented by $v$), and its image under $v\mapsto(\ct(\bar v),\alpha(t)^{-1}\bar v)$ in phase 2. In particular, the partner component is nonempty at time $t$ if and only if $T$ has a node at level $t$.

If $\alpha\in E_{n,F}$, a witness $\beta$ has all its prefixes in $T$, so $T$ has nodes at every level. Conversely if $T$ has nodes at every level then, being an infinite finitely branching tree, it has an infinite branch $\beta$ by K\"onig's lemma. All prefixes of $\beta$ are geodesic, so $\beta\in P$; condition (ii) gives $\beta\sim\alpha$; condition (iii) gives $\beta(n)\notin F$. Hence $\alpha\in E_{n,F}$.
\end{proof}

At this point we have established that $E_{n,F}$ is a so-called safety language.  Since it is well-known that the inclusion relation between safety languages is decidable (see, for example, \cite{KV}), Theorem \ref{inclusiondecidable} is now proven.

All of the necessary ingredients have been established and thus the proof of Theorem \ref{maintheorem} is complete.

\appendix

\section{An alternate proof for free groups}\label{cuntz}

Consider the usual embedding of $\mathbb{F}_2$ in $\mathbb{Z}_2*\mathbb{Z}_3$:  if $a$ and $b$ are the generators of $\mathbb{Z}_2$ and $\mathbb{Z}_3$, then $bab^{-1}a$ and $b^{-1}aba$ generate a copy of $\mathbb{F}_2$.  This induces an embedding of $C^*_r(\mathbb{F}_2)$ into $C^*_r(\mathbb{Z}_2*\mathbb{Z}_3)$, which is clearly computable given the explicit nature of the embedding on the level of groups.  So it suffices to show that the norm on $C^*_r(\mathbb{Z}_2*\mathbb{Z}_3)$ is computable.  However, the latter embeds in $\mathcal{O}_2$ in a concrete way by a theorem of Choi \cite{choi}:  if $S_1$ and $S_2$ denote the canonical generators of $\mathcal{O}_2$, then mapping $a$ to $S_2S_1^*+S_1S_2^*$ and $b$ to $S_2S_1S_1^*+S_2S_2S_1^*S_2^*+S_2S_2^*S_2^*$ extends to a computable embedding of $C^*_r(\mathbb{Z}_2*\mathbb{Z}_3)$ into $\mathcal{O}_2$; since the latter algebra has computable norm (see \cite[Corollary 3.16]{fox}), we are done.  Since each $\mathbb{F}_n$ admits a computable embedding into $\mathbb{F}_2$, this also yields the same results for $C^*_r(\mathbb{F}_n)$, with $n\in \{2,3,\ldots,\infty\}$.  By embedding $C^*_r(\mathbb{F}_2\times \mathbb{F}_2)$ into $\mathcal{O}_2\otimes \mathcal{O}_2$ in the analogous way, this argument can be adapted to establish computability of the standard presentation of $C^*_r(\mathbb{F}_2\times \mathbb{F}_2)$.

\section{The effective metrization theorem}\label{metrizationappendix}

In this appendix, we prove Theorem \ref{thm:metrize}.  Throughout this appendix, $X$ is computably compact and computably Hausdorff, with basis $(B_i)_{i\in \N}$, inclusion relation $\Incl$, and Hausdorff witness set $W$. Overtness is not assumed until later when it is actually needed to produce the set of special points of the presentation.

By a \emph{formal open} we simply mean a finite set $A\subseteq\N$, meant to denote $B_A:=\bigcup_{i\in A}B_i$.  Note that $B_\emptyset=\emptyset$.

For $i,k\in \N$, set $i\le k$ if and only if $(k,j,i)\in\Incl$ or $(j,k,i)\in\Incl$ for some $j\in \N$, and let $\le^\ast$ be the reflexive transitive closure of $\leq$. Note that both relations are c.e. and that $i\le^\ast k$ implies $B_i\subseteq B_k$.

\begin{lemma}\label{lem:refine}
Given $k_1,\dots,k_p\in \N$ and a point $x\in B_{k_1}\cap\cdots\cap B_{k_p}$, one can find $i\in \N$ with $x\in B_i$ and $i\le^\ast k_t$ for every $t$. The search terminates whenever such an $x$ exists.
\end{lemma}

\begin{proof}
Richness of $\Incl$ applied to $x\in B_{k_1}\cap B_{k_1}$ gives $m_1$ with $x\in B_{m_1}$ and $m_1\le k_1$. Inductively, richness applied to $x\in B_{m_{t-1}}\cap B_{k_t}$ gives $m_t$ with $x\in B_{m_t}$, $m_t\le m_{t-1}$ and $m_t\le k_t$. Then $i:=m_p$ satisfies $i\le^\ast k_t$ for every $t$ by transitivity.
\end{proof}

For $i,j\in \N$, set $\mathrm{Ap}(i,j)$ if and only if there exists $(k,l)\in W$ with $i\le^\ast k$ and $j\le^\ast l$.
This relation is c.e. and $\mathrm{Ap}(i,j)$ implies $B_i\cap B_j=\emptyset$.  The relation $\mathrm{Ap}(\cdot,\cdot)$ is \emph{downward closed} in each argument: if $\mathrm{Ap}(i,j)$ and $i'\le^\ast i$, $j'\le^\ast j$, then $\mathrm{Ap}(i',j')$. For formal opens $A$ and $B$, set
\[
\mathrm{Ap}(A,E)\quad:\iff\quad \mathrm{Ap}(i,j) \text{ for all } i\in A,\ j\in E,
\]
which is c.e., implies $B_A\cap B_E=\emptyset$, is monotone under unions in $A$, and is downward closed in $E$.

Finally, for formal opens $A$ and $B$, set $A\prec B$ if and only there exists finite $E\subseteq \N$ with $\mathrm{Ap}(A,E)$ and $X=B_B\cup B_E$.
We call such an $E$ a \emph{witness}. The relation $\prec$ is c.e. as we may search over finite $E$, checking $\mathrm{Ap}(A,E)$ and, by computable compactness, if $X=B_B\cup B_E$.

\begin{lemma}\label{lem:prec-down}
If $A\prec B$ with witness $E$, and $A'$ is a formal open such that every $i'\in A'$ satisfies $i'\le^\ast i$ for some $i\in A$, then $E$ also witnesses $A'\prec B$.
\end{lemma}

\begin{proof}
$\mathrm{Ap}(A,E)$ gives $\mathrm{Ap}(i,j)$ for $i\in A$, $j\in E$, and $\mathrm{Ap}(\cdot,\cdot)$ is downward closed in its first argument, so $\mathrm{Ap}(i',j)$.
\end{proof}

\begin{lemma}\label{lem:prec-sound}
If $E$ witnesses $A\prec B$ then $\overline{B_A}\subseteq X\setminus B_E\subseteq B_B$.
\end{lemma}

\begin{proof}
By definition, we have $B_A\cap B_E=\emptyset$, so $B_A\subseteq X\setminus B_E$, which is closed, whence $\overline{B_A}\subseteq X\setminus B_E$. Since $X=B_B\cup B_E$, we have $X\setminus B_E\subseteq B_B$.
\end{proof}

Note that $\prec$ is only a \emph{subrelation} of the relation $\overline{B_A}\subseteq B_B$ on formal opens; this is intentional as the full relation is $\Pi^0_1$ and hence there is no prospect of effectively enumerating it. What must be shown is that the subrelation is large enough for our purposes, which is the content of the next three lemmas.

\begin{lemma}\label{lem:rich}
Let $B$ be a formal open and $x\in B_B$. Then there is $i$ with $x\in B_i$ and $\{i\}\prec B$.
\end{lemma}

\begin{proof}
For each $y\in X\setminus B_B$, there is $(k_y,l_y)\in W$ with $x\in B_{k_y}$, $y\in B_{l_y}$. The sets $B_{l_y}$ cover the compact set $X\setminus B_B$, so finitely many suffice: say $E:=\{l_1,\dots,l_p\}$ with $X=B_B\cup B_E$, and with associated $k_1,\dots,k_p$ all containing $x$. (If $X=B_B$ take $E=\emptyset$ and any $i$ with $x\in B_i$.)

By Lemma~\ref{lem:refine}, there is $i$ with $x\in B_i$ and $i\le^\ast k_t$ for all $t$. Since $(k_t,l_t)\in W$ and $l_t\le^\ast l_t$, we get $\mathrm{Ap}(i,l_t)$ for every $t$, that is $\mathrm{Ap}(\{i\},E)$. It follows that $E$ witnesses $\{i\}\prec B$.
\end{proof}

\begin{lemma}\label{lem:union}
From witnesses for $A_1\prec B,\dots,A_r\prec B$, one can compute a witness for $A_1\cup\cdots\cup A_r\prec B$.
\end{lemma}

\begin{proof}
Let $E_s$ witness $A_s\prec B$. Then $X\setminus B_B\subseteq B_{E_s}$ for each $s$, so $$X\setminus B_B\subseteq\bigcap_s B_{E_s}=\bigcup_{\vec e\,\in\,E_1\times\cdots\times E_r}\ \bigcap_{s}B_{e_s}.$$

By Lemma~\ref{lem:refine}, each $\bigcap_s B_{e_s}$ is the union of those basic $B_m$ with $m\le^\ast e_s$ for all $s$. Hence the c.e.\ family
\[
\mathcal G:=\big\{\,m\in \N\ :\ \text {there is }\vec e\in E_1\times\cdots\times E_r \text{ with } m\le^\ast e_s \text{ for all } s\,\big\}
\]
has $\bigcup_{m\in\mathcal G}B_m=\bigcap_s B_{E_s}\supseteq X\setminus B_B$. Since $X\setminus B_B$ is effectively closed and $\bigcup_{m\in\mathcal G}B_m$ effectively open, computable compactness produces a finite $E\subseteq\mathcal G$ with $X=B_B\cup B_E$.

It remains to check $\mathrm{Ap}(A_1\cup\cdots\cup A_r,E)$. Let $i\in A_s$ and $m\in E$, and let $\vec e$ witness $m\in\mathcal G$. Then $\mathrm{Ap}(i,e_s)$ because $\mathrm{Ap}(A_s,E_s)$ and $e_s\in E_s$; since $m\le^\ast e_s$, downward closure of $\mathrm{Ap}(\cdot,\cdot)$ in its second argument gives $\mathrm{Ap}(i,m)$.
\end{proof}

\begin{lemma}\label{lem:interp}
From a witness for $A\prec B$, one can compute a formal open $C$ and witnesses for
\[
A\prec C\prec B .
\]
\end{lemma}

\begin{proof}
Let $E$ witness $A\prec B$ and put $Z:=X\setminus B_E$, an effectively closed set, with $Z\subseteq B_B$ by Lemma~\ref{lem:prec-sound}.

The family $\{B_i:\{i\}\prec B\}$ is effectively enumerable, its index set being c.e.\ because $\prec$ is c.e.. It covers $Z$ as any $z\in Z$ lies in $B_B$, so Lemma~\ref{lem:rich} supplies a suitable $i$ with $\{i\}\prec B$ and $z\in B_i$. By computable compactness, one can compute a finite $C$ with $Z\subseteq B_C$ and $\{i\}\prec B$ for each $i\in C$.

It follows that $C\prec B$, with a witness computed by Lemma~\ref{lem:union}. Note also that $E$ itself witnesses $A\prec C$: we already have $\mathrm{Ap}(A,E)$, and $X=B_C\cup B_E$ because $X\setminus B_E=Z\subseteq B_C$.
\end{proof}

We next prove an effective version of Urysohn's lemma in our context.

\begin{lemma}\label{lem:urysohn}
From a witness for $A\prec B$, one can find $f:X\to[0,1]$ with $f\equiv0$ on $B_A$ and $f\equiv1$ off $B_B$, such that $\{f<q\}$ and $\{f>q\}$ are effectively open uniformly in the witness for $A\prec B$ and $q\in\Q$.
\end{lemma}

\begin{proof}
Set $A_0:=A$ and $A_1:=B$ and, using Lemma~\ref{lem:interp} repeatedly, compute formal opens $A_q$ for all dyadic rationals $q\in[0,1]$ with $A_q\prec A_{q'}$, and witnesses $E(q,q')$, whenever $q<q'$. Define
\[
f(x):=\inf\{q : x\in B_{A_q}\}.
\]
By Lemma~\ref{lem:prec-sound}, the $B_{A_q}$ increase with $q$, so $f$ is well defined, vanishes on $B_{A_0}=B_A$, and equals $1$ off $B_{A_1}=B_B$.

Since $\{f<q\}=\bigcup_{q'<q}B_{A_{q'}}$ and
\[
\{f>q\}=\bigcup_{q<q'<q''}B_{E(q',q'')} ,
\]
both are clearly effectively open. To see the latter equality, note that if $x\in B_{E(q',q'')}$, then $x\notin B_{A_{q'}}$ by $\mathrm{Ap}(A_{q'},E(q',q''))$, so $f(x)\ge q'>q$.  Conversely, if $f(x)>q$, choose $q<q'<q''<f(x)$, so $x\notin B_{A_{q''}}$ and hence $x\in B_{E(q',q'')}$ because $X=B_{A_{q''}}\cup B_{E(q',q'')}$.
\end{proof}

Let $(f_n)_{n\in\N}$ effectively enumerate the functions of Lemma~\ref{lem:urysohn} over the c.e.\ set of pairs $(A,B)$ with $A\prec B$, and set
\[
\rho(x,y):=\sum_n 2^{-n}\,|f_n(x)-f_n(y)| .
\]

\begin{proposition}\label{prop:metricfn}
$\rho$ is a computable metric on $X$.
\end{proposition}

\begin{proof}
Since the series converges effectively, $\rho$ is computable.  It is clearly a pseudometric.  To see that it separates points, suppose $x\ne y$ and take $(j,l)\in W$ with $x\in B_j$, $y\in B_l$; Lemma~\ref{lem:rich} applied to $x\in B_{\{j\}}$ gives $i$ with $x\in B_i$ and $\{i\}\prec\{j\}$, and the corresponding $f$ has $f(x)=0$, $f(y)=1$ since $y\notin B_j$. Hence $\rho(x,y)>0$.
\end{proof}

At this point we have produced the metric for our computable metric space presentation of $X$.  Producing the special points for our presentation requires a nested sequence of basic sets, and it requires their $\rho$-diameters to shrink. The following makes small basic sets recognisable, which is the one point where the explicit form of $\rho$ is used rather than just its computability.

\begin{lemma}\label{lem:small}
For each $n\in \N$ and each pair $a<b$ of rational numbers, let $W_{n,a,b}$ be a c.e.\ index set with $\bigcup_{i\in W_{n,a,b}}B_i=f_n^{-1}(a,b)$, closed under refinement in the sense that $i'\le^\ast i\in W_{n,a,b}$ implies $i'\in W_{n,a,b}$. Let $\mathrm{Small}(k)$ denote the set of $i\in \N$ such that, for each $n\leq k$, there are rational numbers $a_n<b_n$ such that $b_n-a_n<2^{-k}$ and $i\in W_{n,a_n,b_n}$.
Then:
\begin{itemize}
    \item $\mathrm{Small}(k)$ is c.e.\ and closed under refinement.
    \item If $i\in\mathrm{Small}(k)$, then $\diam_\rho(\overline{B_i})\le2^{-k+2}$.
    \item For every $x\in X$, there is $i\in\mathrm{Small}(k)$ with $x\in B_i$.
\end{itemize}
\end{lemma}

\begin{proof}
We first note that such $W_{n,a,b}$ exist: take any c.e.\ index set for the effectively open $f_n^{-1}(a,b)$ (available by Lemma~\ref{lem:urysohn}) and close it under $\le^\ast$, which changes neither the union nor c.e.-ness. Closure and c.e.-ness of $\mathrm{Small}(k)$ follow.

For the diameter bound, fix $i\in\mathrm{Small}(k)$. Since $f_n$ is continuous, we have $$\overline{B_i}\subseteq\overline{f_n^{-1}(a_n,b_n)}\subseteq f_n^{-1}[a_n,b_n],$$ so $f_n$ varies by at most $2^{-k}$ on $\overline{B_i}$ for each $n\le k$. Hence, for $x,y\in\overline{B_i}$, we have
\[
\rho(x,y)\le\ 2^{-k}\sum_{n\le k}2^{-n}+\sum_{n>k}2^{-n}\ \le\ 2^{-k+1}+2^{-k}\ \le\ 2^{-k+2}.
\]

For the last point, fix $x\in X$ and $k\in \N$. For each $n\le k$ continuity gives rationals $a_n<f_n(x)<b_n$ with $b_n-a_n<2^{-k}$, and then some $j_n\in W_{n,a_n,b_n}$ with $x\in B_{j_n}$. Lemma~\ref{lem:refine} produces $i$ with $x\in B_i$ and $i\le^\ast j_n$ for every $n\le k$; by closure under refinement, $i\in W_{n,a_n,b_n}$ for each $n\le k$, establishing that $i\in\mathrm{Small}(k)$.
\end{proof}

We are now ready to prove Theorem \ref{thm:metrize}:

\begin{proof}[Proof of Theorem \ref{thm:metrize}]
Let $\Omega=\{i\in\N:B_i\ne\emptyset\}$ be the c.e.\ set given by overtness and let $\rho$ be as in Proposition~\ref{prop:metricfn}.

Let $i_0\in\Omega$ be given. Build indices $i_1,i_2,\dots$ such that, for every $m\ge1$, we have
\[
\{i_m\}\prec\{i_{m-1}\},\qquad i_m\in\mathrm{Small}(m),\qquad i_m\in\Omega ,
\]
together with witnesses. At stage $m$, pick $x\in B_{i_{m-1}}$ (nonempty, as $i_{m-1}\in\Omega$). Lemma~\ref{lem:rich} gives $i'$ with $x\in B_{i'}$ and $\{i'\}\prec\{i_{m-1}\}$; Lemma~\ref{lem:small} gives $j\in\mathrm{Small}(m)$ with $x\in B_j$; Lemma~\ref{lem:refine} gives $i_m$ with $x\in B_{i_m}$, $i_m\le^\ast i'$ and $i_m\le^\ast j$. Then $i_m\in\mathrm{Small}(m)$ by closure under refinement, $\{i_m\}\prec\{i_{m-1}\}$ by Lemma~\ref{lem:prec-down}, and $i_m\in\Omega$ since $x\in B_{i_m}$, this last being confirmed by the overtness enumeration after finitely many steps. So all three conditions are simultaneously achievable, and the search for a triple satisfying them terminates.

By Lemma~\ref{lem:prec-sound}, $\overline{B_{i_m}}\subseteq B_{i_{m-1}}$, so the sets $\overline{B_{i_m}}$ are nested; each is nonempty since $i_m\in\Omega$; and by compactness of $X$ we have $\bigcap_m\overline{B_{i_m}}\ne\emptyset$. By Lemma~\ref{lem:small}, $\diam_\rho(\overline{B_{i_m}})\le2^{-m+2}$, whence the intersection is a singleton, which we denote by $x^\ast(i_0)$. Note $x^\ast(i_0)\in\overline{B_{i_1}}\subseteq B_{i_0}$.

For $n\le m$, membership $i_m\in\mathrm{Small}(m)$ supplies rationals with $b_n-a_n<2^{-m}$ and $\overline{B_{i_m}}\subseteq f_n^{-1}[a_n,b_n]$, so $f_n(x^\ast)\in[a_n,b_n]$. Thus $f_n(x^\ast)$ is computable uniformly in $n$ and in the construction, and therefore so is
\[
\rho\big(x^\ast(i_0),x^\ast(i_0')\big)=\sum_n2^{-n}\big|f_n(x^\ast(i_0))-f_n(x^\ast(i_0'))\big| ,
\]
as the series converges effectively.

Let $(\alpha_r)_{r\in \N}$ enumerate $x^\ast(i_0)$ as $i_0$ ranges over $\Omega$. Any nonempty open $U$ contains a basic $B_{i_0}$ which is nonempty, hence in $\Omega$, and $x^\ast(i_0)\in B_{i_0}\subseteq U$, proving that the sequence is dense

The family $\{B_i : i\in\mathrm{Small}(k)\cap\Omega\}$ is c.e. and covers $X$ by Lemma~\ref{lem:small}. Computable compactness yields a finite subcover indexed by $i^{(1)},\dots,i^{(r)}\in\mathrm{Small}(k)\cap\Omega$. Each $B_{i^{(t)}}$ has $\rho$-diameter at most $2^{-k+2}$ and contains the special point $x^\ast(i^{(t)})$, so the balls $B_\rho(x^\ast(i^{(t)}),2^{-k+2})$ cover $X$. As $k$ is arbitrary this establishes computable compactness of the the computable metric space $(X,\rho,(\alpha_r))$.

The identity $(X,\tau)\to(X,\rho)$ is computable (and thus continuous) by Proposition \ref{prop:metricfn}. Its inverse is computable (and thus continuous) by Corollary~\ref{cor:inverse}, $X$ being computably compact and $(X,\rho)$ computably Hausdorff.
\end{proof}

\end{document}